\documentclass[11pt]{article}

\usepackage[T1]{fontenc}
\usepackage{lmodern}
\usepackage{microtype}
\usepackage[a4paper,margin=30mm]{geometry}
\usepackage{amsmath,amssymb,amsthm,mathtools}
\usepackage{enumitem}
\usepackage[hidelinks]{hyperref}

\hypersetup{
  pdftitle={Arithmetic classification of the Dirichlet eigenvalues for a one-dimensional nonlocal p-Laplacian problem},
  pdfauthor={Shingo Takeuchi},
  pdfkeywords={p-Laplacian, nonlinear eigenvalue, Dirichlet eigenvalues, beta function, transcendence}
}

\numberwithin{equation}{section}
\allowdisplaybreaks

\newtheorem{theorem}{Theorem}[section]

\newtheorem{corollary}[theorem]{Corollary}
\theoremstyle{remark}
\newtheorem{remark}[theorem]{Remark}
\newtheorem{example}[theorem]{Example}

\newcommand{\Qbar}{\overline{\mathbb{Q}}}
\newcommand{\Qbarp}{\overline{\mathbb{Q}}_{>0}}
\newcommand{\R}{\mathbb{R}}
\newcommand{\N}{\mathbb{N}}
\newcommand{\norm}[1]{\left\lVert #1\right\rVert}
\newcommand{\abs}[1]{\left\lvert #1\right\rvert}
\DeclareMathOperator{\Beta}{B}

\title{Arithmetic classification of the Dirichlet eigenvalues for a one-dimensional nonlocal \texorpdfstring{$p$}{p}-Laplacian problem}
\author{Shingo Takeuchi\\[2mm]
\small Department of Mathematical Sciences, Shibaura Institute of Technology}
\date{}

\begin{document}
\maketitle

\begin{abstract}
We classify the arithmetic nature of the Dirichlet eigenvalues of a homogeneous
nonlocal one-dimensional $p$-Laplacian problem with power-type nonlinearity
$\abs{u}^{q-2}u$. For rational $p,q>1$, either all eigenvalues are algebraic or all
are transcendental. Algebraicity holds exactly
when the interval length is an algebraic multiple of a distinguished scale
determined by a beta value. This scale is transcendental. When \mbox{$p=q$}, the
criterion reduces to algebraicity of the interval length divided by $\pi$,
independently of the rational exponent.
\end{abstract}



\section{Introduction}

Let $p,q>1$ and let $I=(a,b)$ be an interval of length $\ell=b-a$. We consider
\begin{equation}\label{eq:problem}
 -\bigl(\abs{u'}^{p-2}u'\bigr)'
 =\lambda\norm{u}_{L^q(I)}^{p-q}\abs{u}^{q-2}u
 \quad\text{in }I,
 \qquad u(a)=u(b)=0,
\end{equation}
where
\[
 \norm{u}_{L^q(I)}
 :=\left(\int_a^b\abs{u(x)}^q\,dx\right)^{1/q}.
\]
An eigenfunction is a nonzero weak solution in $W^{1,p}_0(I)$. The nonlocal
factor makes the equation invariant under multiplication of $u$ by any nonzero
constant.

The known eigenvalue formula involves the generalized half-period
\begin{equation}\label{eq:pipq-intro}
 \pi_{p,q}:=2\int_0^1(1-t^q)^{-1/p}\,dt
 =\frac2q\Beta\!\left(1-\frac1p,\frac1q\right),
\end{equation}
where $\Beta$ denotes Euler's beta function. For rational $p,q>1$, set
\begin{equation}\label{eq:r-intro}
 r_{p,q}:=\frac{pq}{pq+p-q}.
\end{equation}
The purpose of this note is to prove that an eigenvalue is algebraic if and
only if $\ell/\pi_{p,q}^{r_{p,q}}$ is algebraic. The criterion is independent of
the index, so either all eigenvalues are algebraic or all are transcendental.
Schneider's theorem~\cite{Schneider1941} shows that the distinguished scale
$\pi_{p,q}^{r_{p,q}}$ is transcendental. Hence algebraic lengths give only
transcendental eigenvalues, whereas a countable dense set of transcendental
lengths gives algebraic eigenvalues. For $p=q$, the criterion reduces to
algebraicity of $\ell/\pi$, independently of both the index and the rational
exponent.

The analytic ingredients are known. Dr\'abek and
Man\'asevich~\cite{DrabekManasevich1999} obtained explicit solutions for the
corresponding nonhomogeneous problem. Franzina and
Lamberti~\cite{FranzinaLamberti2010} recorded the complete eigenvalue
sequence for \eqref{eq:problem}; its dependence on the parameters was studied
by Kajikiya, M.~Tanaka and S.~Tanaka~\cite{KajikiyaTanakaTanaka2025}.
We recall the eigenvalue formula in Section~2, treat its beta value in
Section~3, and establish the arithmetic classification and its diagonal
consequence in Sections~4 and~5.

\section{The explicit Dirichlet eigenvalues}

Write $\N:=\{1,2,\ldots\}$ and $p^*:=p/(p-1)$. Let
\[
 F_{p,q}(x):=\int_0^x(1-t^q)^{-1/p}\,dt,\qquad 0\le x\le1.
\]
Its inverse on $[0,\pi_{p,q}/2]$, extended by reflection and odd
$2\pi_{p,q}$-periodicity, is denoted by $\sin_{p,q}$. For background, see Kobayashi and
Takeuchi~\cite{KobayashiTakeuchi2019}. The function $u(x):=\sin_{p,q}x$ satisfies
\begin{equation}\label{eq:sinpq-ode}
 -\bigl(\abs{u'}^{p-2}u'\bigr)'
 =\frac q{p^*}\abs{u}^{q-2}u,
 \qquad x\in\R.
\end{equation}

\begin{theorem}[Known eigenvalue formula]\label{thm:eigenvalues}
Let $p,q>1$ and let $I=(a,b)$ have length $\ell=b-a$. The eigenvalues of
\eqref{eq:problem} are precisely
\begin{equation}\label{eq:lambda-eigenvalues}
 \lambda_n(p,q;\ell)
 =\frac q{p^*}\left(\frac{n\pi_{p,q}}\ell\right)^p
 \left(\frac{p^*+q}{\ell p^*}\right)^{(p-q)/q},
 \qquad n\in\N.
\end{equation}
A corresponding eigenfunction is
\begin{equation}\label{eq:eigenfunction}
 v_n(x):=\sin_{p,q}\!\left(\frac{n\pi_{p,q}}\ell(x-a)\right).
\end{equation}
In particular, all eigenvalues are positive and
\begin{equation}\label{eq:n-factor}
 \lambda_n(p,q;\ell)=n^p\lambda_1(p,q;\ell).
\end{equation}
\end{theorem}

\begin{proof}[Reference and normalization check]
The complete eigenvalue sequence is given by Franzina and
Lamberti~\mbox{\cite[Theorem~5.3]{FranzinaLamberti2010}}, based on the one-dimensional
results of Dr\'abek and Man\'asevich~\cite{DrabekManasevich1999}. The quantity
called $\pi_{pq}$ in the former reference equals $(q/2)\pi_{p,q}$ in our
notation. We verify the constants in \eqref{eq:lambda-eigenvalues} directly.

Using \cite[(21)]{KobayashiTakeuchi2019} with the integer index
in that formula set to $1$, together with reflection and periodicity,
we obtain
\[
 \begin{aligned}
 \norm{v_n}_{L^q(I)}^q
 &=\frac{2\ell}{\pi_{p,q}}
   \int_0^{\pi_{p,q}/2}\sin_{p,q}^{q} y\,dy\\
 &=\frac{2\ell}{\pi_{p,q}}
   \frac{1/q}{1/p^*+1/q}\,\frac{\pi_{p,q}}2
 =\frac{\ell p^*}{p^*+q}.
 \end{aligned}
\]
On the other hand, \eqref{eq:sinpq-ode} implies
\[
 -\bigl(\abs{v_n'}^{p-2}v_n'\bigr)'
 =\frac q{p^*}\left(\frac{n\pi_{p,q}}\ell\right)^p
  \abs{v_n}^{q-2}v_n.
\]
Thus the eigenvalue associated with $v_n$ is
$\frac q{p^*}(n\pi_{p,q}/\ell)^p\norm{v_n}_{L^q(I)}^{q-p}$, which is
\eqref{eq:lambda-eigenvalues}. The absence of further eigenvalues is part of
the cited theorem, and \eqref{eq:n-factor} follows from the formula.
\end{proof}

\section{Transcendence of the generalized half-period}

We use the following classical consequence of
Schneider's theory~\cite{Schneider1941}: if $x,y$ are positive rational
nonintegers and $x+y\notin\mathbb Z$, then $\Beta(x,y)$ is transcendental.
Masser~\cite[p.~234]{Masser1977} also records the conclusion for positive
rational nonintegral beta arguments.

\begin{theorem}\label{thm:pipq-trans}
If $p,q>1$ are rational, then $\pi_{p,q}$ is transcendental.
\end{theorem}

\begin{proof}
Put $x=1-1/p$ and $y=1/q$, so $x,y\in\mathbb Q\cap(0,1)$. If $p\ne q$, then
$0<x+y<2$ and $x+y\ne1$; hence Schneider's theorem applies. If $p=q$, then
$y=1-x$, and Euler's reflection formula gives
\begin{equation}\label{eq:reflection}
 \Beta(x,1-x)=\Gamma(x)\Gamma(1-x)=\frac\pi{\sin(\pi x)},
\end{equation}
where $\Gamma$ is Euler's gamma function. Since $x$ is rational,
$e^{i\pi x}$ is a root of unity, so $\sin(\pi x)$ is a nonzero algebraic
number. The right-hand side is therefore transcendental. In both cases,
\eqref{eq:pipq-intro} yields the conclusion.
\end{proof}

\begin{remark}
The only integral-sum case for the arguments $1-1/p$ and $1/q$ is $p=q$,
which explains the separate use of the reflection formula. More generally, if
$x,y>0$ are rational nonintegers with $x+y=N\in\N$, the gamma recurrence gives
\[
 \Beta(x,y)
 =\frac{\prod_{k=1}^{N-1}(k-x)}{(N-1)!}
  \frac\pi{\sin(\pi x)}.
\]
The product is empty when $N=1$. All factors other than $\pi$ are nonzero
algebraic numbers. Thus the integral-sum case is also transcendental, consistent
with the formulation recorded by Masser.
\end{remark}

\section{Complete arithmetic classification of the eigenvalues}

We write $\mathbb Q_{>0}:=\mathbb Q\cap(0,\infty)$ for the positive rational
numbers, $\Qbar$ for the field of algebraic numbers, and
$\Qbarp:=\Qbar\cap(0,\infty)$ for the positive algebraic numbers.

For rational $p,q>1$, recall $r_{p,q}$ from \eqref{eq:r-intro} and set
\[
 \tau_{p,q}:=\pi_{p,q}^{r_{p,q}}.
\]

\begin{theorem}[Arithmetic classification for fixed parameters]\label{thm:classification}
Fix rational numbers $p,q>1$, and let $I=(a,b)$ have length $\ell=b-a$.
The following statements are equivalent:
\begin{enumerate}[label=\textup{(\roman*)},itemsep=1pt,topsep=4pt]
 \item $\ell/\tau_{p,q}$ is algebraic;
 \item $\lambda_n(p,q;\ell)$ is algebraic for some $n\in\N$;
 \item $\lambda_n(p,q;\ell)$ is algebraic for every $n\in\N$.
\end{enumerate}
Likewise, the following statements are equivalent:
\begin{enumerate}[label=\textup{(\roman*)},start=4,itemsep=1pt,topsep=4pt]
 \item $\ell/\tau_{p,q}$ is transcendental;
 \item $\lambda_n(p,q;\ell)$ is transcendental for some $n\in\N$;
 \item $\lambda_n(p,q;\ell)$ is transcendental for every $n\in\N$.
\end{enumerate}
In particular, for the fixed pair $(p,q)$, the interval lengths producing
algebraic eigenvalues form exactly the set
\[
 \mathcal A_{p,q}
 :=\{\ell>0:\lambda_n(p,q;\ell)\in\Qbar\text{ for every }n\in\N\}
 =\tau_{p,q}\Qbarp.
\]
\end{theorem}

\begin{proof}
For $z>0$ and $s\in\mathbb Q\setminus\{0\}$, algebraicity of $z$ is equivalent
to algebraicity of $z^s$: rational powers of positive algebraic numbers are
algebraic, and the converse follows by taking the power $1/s$.

Set
\[
 \theta_{p,q}:=\frac{pq+p-q}{q},\qquad
 \kappa_{p,q}:=\frac q{p^*}
  \left(\frac{p^*+q}{p^*}\right)^{(p-q)/q}.
\]
Then $r_{p,q},\theta_{p,q}\in\mathbb Q_{>0}$,
$\kappa_{p,q}\in\Qbarp$ and $r_{p,q}\theta_{p,q}=p$.
Fix $n\in\N$. Since $\tau_{p,q}=\pi_{p,q}^{r_{p,q}}$, the eigenvalue formula
\eqref{eq:lambda-eigenvalues} gives
\begin{equation}\label{eq:arithmetic-factorization}
 \lambda_n(p,q;\ell)
 =n^p\kappa_{p,q}\pi_{p,q}^p\ell^{-\theta_{p,q}}
 =n^p\kappa_{p,q}
  \left(\frac{\ell}{\tau_{p,q}}\right)^{-\theta_{p,q}}.
\end{equation}
The coefficient $n^p\kappa_{p,q}$ is positive and algebraic. The observation
above therefore yields
\[
 \begin{aligned}
 \lambda_n(p,q;\ell)\in\Qbar
 &\quad\Longleftrightarrow\quad
 \left(\frac{\ell}{\tau_{p,q}}\right)^{-\theta_{p,q}}\in\Qbar\\
 &\quad\Longleftrightarrow\quad
 \frac{\ell}{\tau_{p,q}}\in\Qbar.
 \end{aligned}
\]
The condition on the right is independent of $n$, proving
\textup{(i)}--\textup{(iii)}. Taking complements gives the transcendental
equivalence for each $n$, and the same independence proves
\textup{(iv)}--\textup{(vi)}. Finally, \textup{(i)} says precisely that
$\ell\in\tau_{p,q}\Qbarp$, which identifies $\mathcal A_{p,q}$.
\end{proof}

The all-or-nothing conclusion also follows directly from
\eqref{eq:n-factor}, since $n^p$ is algebraic for rational $p$.
The transcendence result of Section~3 further determines the arithmetic
nature of the exceptional lengths.

\begin{corollary}[Algebraic and exceptional interval lengths]\label{cor:exceptional-size}
Fix rational $p,q>1$. Every algebraic length $\ell>0$ gives only transcendental
eigenvalues. The set $\mathcal A_{p,q}$ is countable and dense in $(0,\infty)$,
and all its elements are transcendental. Its complement has full Lebesgue
measure and consists exactly of the lengths giving only transcendental
eigenvalues.
\end{corollary}

\begin{proof}
By Theorem~\ref{thm:pipq-trans}, $\pi_{p,q}$ is transcendental. Since
$r_{p,q}\in\mathbb Q_{>0}$, the number $\tau_{p,q}$ is also transcendental:
otherwise $\pi_{p,q}=\tau_{p,q}^{1/r_{p,q}}$ would be algebraic.
Every nonzero algebraic multiple of $\tau_{p,q}$ is therefore
transcendental. Since
$\mathcal A_{p,q}=\tau_{p,q}\Qbarp$, no algebraic length belongs to
$\mathcal A_{p,q}$. The set $\Qbarp$ is countable and dense in $(0,\infty)$,
and multiplication by $\tau_{p,q}>0$ preserves both properties. The remaining
claims follow from Theorem~\ref{thm:classification} and the fact that a
countable set has Lebesgue measure zero.
\end{proof}

We retain the notation $\theta_{p,q}$ and $\kappa_{p,q}$ introduced in the
proof of Theorem~\ref{thm:classification}. For fixed rational $p,q>1$ and
$n\in\N$, every $\beta\in\Qbarp$ can also be realized as $\lambda_n(p,q;\ell)$
at the unique length
\begin{equation}\label{eq:inverse-length}
 \ell=\tau_{p,q}
 \left(\frac{n^p\kappa_{p,q}}{\beta}\right)^{1/\theta_{p,q}}.
\end{equation}
This follows by solving \eqref{eq:arithmetic-factorization}; the length is
transcendental, and uniqueness follows from $\theta_{p,q}>0$.

\begin{example}\label{ex:non-diagonal}
For $(p,q)=(2,4)$,
\[
 \pi_{2,4}=\tfrac12\Beta(\tfrac12,\tfrac14),\qquad
 r_{2,4}=\tfrac43,\qquad \theta_{2,4}=\tfrac32,\qquad
 \kappa_{2,4}=\tfrac2{\sqrt3}.
\]
Thus
\[
 \lambda_n(2,4;\ell)
 =\frac{2n^2\pi_{2,4}^2}{\sqrt3\,\ell^{3/2}},
 \qquad \mathcal A_{2,4}=\pi_{2,4}^{4/3}\Qbarp.
\]
\end{example}

\section{The diagonal case}

Put $\pi_p:=\pi_{p,p}$. Euler's reflection formula gives
\begin{equation}\label{eq:pip}
 \pi_p=\frac{2\pi}{p\sin(\pi/p)}=a_p\pi,
 \qquad a_p:=\frac{2}{p\sin(\pi/p)}.
\end{equation}
For rational $p>1$, the coefficient $a_p$ is positive and algebraic.

\begin{theorem}[Diagonal criterion]\label{thm:diagonal}
Let $p>1$ be rational, $\ell>0$ and $n\in\N$. Then
\begin{equation}\label{eq:diagonal-eigenvalues}
 \lambda_n(p,p;\ell)=(p-1)\left(\frac{n\pi_p}{\ell}\right)^p,
\end{equation}
and
\begin{equation}\label{eq:diag-alg}
 \lambda_n(p,p;\ell)\in\Qbar
 \quad\Longleftrightarrow\quad \frac{\ell}{\pi_p}\in\Qbar
 \quad\Longleftrightarrow\quad \frac{\ell}{\pi}\in\Qbar.
\end{equation}
The same equivalences hold with ``algebraic'' replaced by ``transcendental.''
\end{theorem}

\begin{proof}
Set $q=p$ in Theorem~\ref{thm:eigenvalues} and note that
$r_{p,p}=1$, $\theta_{p,p}=p$ and $\tau_{p,p}=\pi_p$.
Theorem~\ref{thm:classification} gives the first equivalence in
\eqref{eq:diag-alg}; the second follows from $\pi_p=a_p\pi$ with
$a_p\in\Qbarp$. Taking complements gives the transcendental equivalences.
\end{proof}

The criterion is independent of both $n$ and $p$, yielding the following
simultaneous statement.

\begin{corollary}[Simultaneous arithmetic behavior]\label{cor:simultaneous}
Fix $\ell>0$. The following statements are equivalent:
\begin{enumerate}[label=\textup{(\roman*)},itemsep=1pt,topsep=4pt]
 \item $\ell/\pi$ is algebraic;
 \item $\lambda_n(p,p;\ell)$ is algebraic for some rational $p>1$ and some $n\in\N$;
 \item $\lambda_n(p,p;\ell)$ is algebraic for every rational $p>1$ and every $n\in\N$.
\end{enumerate}
Likewise, the following statements are equivalent:
\begin{enumerate}[label=\textup{(\roman*)},start=4,itemsep=1pt,topsep=4pt]
 \item $\ell/\pi$ is transcendental;
 \item $\lambda_n(p,p;\ell)$ is transcendental for some rational $p>1$ and some $n\in\N$;
 \item $\lambda_n(p,p;\ell)$ is transcendental for every rational $p>1$ and every $n\in\N$.
\end{enumerate}
\end{corollary}

\begin{proof}
Apply Theorem~\ref{thm:diagonal} to each rational $p>1$ and each $n\in\N$.
\end{proof}

For $p=q=2$, the formula is the classical one,
$\lambda_n(2,2;\ell)=(n\pi/\ell)^2$. Thus $\ell=1$ gives only transcendental
eigenvalues, while $\ell=\pi$ gives $\lambda_n=n^2$. The preceding corollary
shows that the distinction according to $\ell/\pi$ persists simultaneously
for every rational exponent $p>1$.

\section*{Declaration of AI use}
The author used ChatGPT (OpenAI, including GPT-6 Astra Pro) to assist with
mathematical discussions and manuscript preparation. In particular, the author
learned of Schneider's theorem on the transcendence of beta values through
these discussions. The author reviewed and verified the content and takes full
responsibility for the final manuscript.

\end{document}